\documentclass[11pt]{amsart}

\usepackage{amsmath}
\usepackage{amssymb, mathrsfs}
\usepackage{amsfonts, mathrsfs}
\usepackage{dsfont}
\usepackage{amsthm}
\usepackage{relsize}
\usepackage{enumerate}
\usepackage{verbatim}
\usepackage{amscd}
\usepackage{graphicx}
\usepackage{appendix}
\usepackage{tikz}
\usepackage[colorlinks=true, pdfstartview=FitV, linkcolor=blue, citecolor=red, urlcolor=blue]{hyperref} 
\usepackage{mathtools}

\newtheorem{innercustomgeneric}{\customgenericname}
\providecommand{\customgenericname}{}

\newcommand{\newcustomtheorem}[2]{\newenvironment{#1}[1]
  {\renewcommand\customgenericname{#2}
   \renewcommand\theinnercustomgeneric{##1}\innercustomgeneric}{\endinnercustomgeneric}}

\newcustomtheorem{customthm}{Theorem}

\newcommand{\newcustomlemma}[2]{\newenvironment{#1}[1]
  {\renewcommand\customgenericname{#2}
   \renewcommand\theinnercustomgeneric{##1} \innercustomgeneric}{\endinnercustomgeneric}}

\newcustomlemma{customlemma}{Lemma}

\newcustomlemma{customproposition}{Proposition}

\theoremstyle{plain}
\newtheorem{theorem}{Theorem}[section]

\newtheorem*{example}{Examples}

{}
\newtheorem{lemma}[theorem]{Lemma}

\newtheorem{corollary}[theorem]{Corollary}
\theoremstyle{definition}

\theoremstyle{remark}
\newtheorem{remark}{Remark}

\numberwithin{equation}{section}

\newcommand{\R}{\mathbb{R}}
\newcommand{\Z}{\mathbb{Z}}

\newcommand{\cB}{\mathcal{B}}

\newcommand{\cD}{\mathcal{D}}

\newcommand{\cH}{\mathcal{H}}
\newcommand{\cI}{\mathcal{I}}

\newcommand{\cM}{\mathcal{M}}

\newcommand{\bR}{\mathbb{R}}

\newcommand{\bZ}{\mathbb{Z}}
\newcommand{\bN}{\mathbb{N}}

\newcommand{\q}{\quad}

\newcommand{\dd}{\,\mathrm{d}}

\newcommand{\wh}{\widehat}
\newcommand{\wt}{\widetilde}

\begin{document}

\title[Maximal operators with fractal dimension dilation]
{Maximal operators given by Fourier multipliers with dilation of fractional dimensions}

\author[J. B. Lee]{Jin Bong Lee}
\address[J. B. Lee]{Research Institute of Mathematics, Seoul National University, Seoul 08826, Republic of Korea}
\email{jinblee@snu.ac.kr}

\author[J. Seo]{Jinsol Seo}
\address[J. Seo]{School of Mathematics, Korea Institute for Advanced Study, 85 Hoegiro Dongdaemun-gu, Seoul 02455, Republic of Korea}
\email{seo9401@kias.re.kr}


\begin{abstract}
    In this paper, we investigate $L^p$ bounds of maximal Fourier multiplier operators with dilation of fractional dimensions.
    For Fourier multipliers, we suggest a criterion related to dimensions of dilation sets which guarantees $L^p$ bounds of the maximal operators for each $p$.
    Our criterion covers Mikhlin-type multipliers, multipliers with limited decay, and multipliers with slow decay.
\end{abstract}

\maketitle

\section{Introduction}

Let $m$ be a bounded function.
Then for a Schwartz function $f$, a Fourier multiplier operator $T_m(f)(x)$ and its dilation $T_{m(t\,\cdot\,)}(f)(x)$ are given by
$$
    \widehat{T_m(f)}(\xi) := m(\xi)\widehat{f}(\xi),\quad
    \widehat{T_{m(t\,\cdot\,)}(f)}(\xi) := m(t\xi)\widehat{f}(\xi),
$$
where $\widehat{f}$ denotes the Fourier transform of $f$.
For a set $E\subset (0,\infty)$, we define the following maximal operators whose dilation is taken in $E$:
\begin{align*}
    \mathcal{M}_m^E(f)(x) \coloneqq \sup_{t\in E}\left|T_{m(t)}(f)(x)\right|
\end{align*}
The purpose of this paper is to investigate conditions for $m$ which guarantee $L^p(\bR^d)$ bounds of $\mathcal{M}_m^E$.

For $E:=\{1\}$, we have $\mathcal{M}_m^E = T_m$, hence the following H\"ormander--Mikhlin condition is a sufficient condition for the $L^p(\bR^d)$ bounds of $\mathcal{M}_m^E$:
\begin{align}\label{HM}
    \sup_{j\in\bZ} \| m(2^j\,\cdot\,) \widehat{\psi}(\,\cdot\,)\|_{L_s^2(\bR^d)}<\infty\quad\text{for some}\quad  s>d/2.
\end{align}
Note that $\widehat{\psi}(\xi/2^j)$ is a frequency cut-off function on $|\xi|\sim 2^j$.
When $E:=(0, \infty)$, Christ--Grafakos--Honz\'ik--Seeger \cite{Ch_Gr_Hon_See2005} proved that the H\"ormander--Mikhlin condition \eqref{HM} is not enough for $L^p(\bR^d)$ bounds of $\mathcal{M}_m^E$.
Moreover, they presented the following sufficient conditions for the $L^p(\bR^d)$ bounds (see \cite[Theorem 1.2 and (1.12)]{Ch_Gr_Hon_See2005}):
\begin{customthm}{A}\label{thm_CGHS}
Let $p\in(0,\infty)$ and $E=(0,\infty)$. Suppose $m$ satisfies that there exist $1<q<\infty$ and $s>\frac{d}{r}+\min\{\frac{1}{p},\frac{1}{p'}\}$, $r:=\min\{p,2\}$, such that
\begin{align}\label{CGHS}
    \left(\sum_{j\in\bZ} \| m(2^j\,\cdot\,) \widehat{\psi}(\,\cdot\,)\|_{L_s^r(\bR^d)}^q\right)^{1/q} <\infty\,.
\end{align}
Then $\mathcal{M}_m^E$ is bounded on $L^p(\bR^d)$.

In addition, if \eqref{CGHS} holds for $q>1$, $r=2$, and $s>\frac{d}{2}$, then $\cM_m^E$ is a bounded map from $L^{\infty}$ to $BMO$.
\end{customthm}
\noindent One can understand the condition \eqref{CGHS} is an $\ell^q(\bZ)$ analogue of \eqref{HM}.   
We recommend \cite{Ch_Gr_Hon_See2005, Gr_Hon_See2006} for more refined summability condition than \eqref{CGHS}.

In \cite{LeeSeo_2023}, the authors of the present paper improved the regularity condition of \eqref{CGHS} to $s>d\left|\frac{1}{2} - \frac{1}{p}\right| + \min\{\frac{1}{p}, \frac{1}{2}\}$ provided that Sobolev spaces $L_s^r$ are replaced with Besov spaces $B_{p_0}^s$ where $1/p_0 =|1/2- 1/p|$, and $q$ is fixed by $2$.

When it comes to the case of $E$ whose dimension is between $0$ and $1$,
to the authors' knowledge, there are not many results on abstract Fourier multipliers.
If the Fourier multipliers are given by Fourier transforms of surface-carried measures, however, one can find remarkable results on $L^p$ bounds of associated maximal averages. 
One of the most interesting objects is the spherical maximal functions.
$$
    \mathcal{M}_{sph}(f)(x) := \sup_{0<t}\left| \int_{\mathbb{S}^{d-1}} f(x-ty)~\mathrm{d}\sigma(y)\right|.
$$
It is widely known due to E. M. Stein \cite{St1976} ($d\geq3$) and J. Bourgain \cite{Bo1986} ($d=2$) that $\mathcal{M}_{sph}$ is bounded on $L^p(\bR^d)$ if and only if $p>\frac{d}{d-1}$.
If one takes $t$ as a lacunary sequence, then $\mathcal{M}_{sph}$ is bounded on $L^p(\bR^d)$ for $p>1$ due to C. P. Calder\'on \cite{CPCal1979}.

If one considers $E$ with dimension strictly between $0$ and $1$, Seeger--Wainger--Wright \cite{SWW1995} showed that $\mathcal{M}_{sph}^E$ is bounded on $L^p(\bR^d)$ if $p>1+\frac{\kappa(E)}{d-1}$, and unbounded on $L^p(\bR^d)$ if $1<p<1+\frac{\kappa(E)}{d-1}$.
The quantity $\kappa(E)$ is given by
\begin{align*}
\kappa(E):=\lim_{\delta\rightarrow 0}\sup_{j\in\bZ}\frac{\log N(E_j,\delta)}{-\log \delta}\,,
\end{align*}
where $E_j:=\big(2^{-j}E\big)\cap [1,2]$, and $N(E, \delta)$ is the $\delta$-entropy number of $E$ defined by
\begin{align*}
    N(E,\delta):=\#\big\{k\in\bN\,:\,E\cap [k\delta,(k+1)\delta]\neq \emptyset\big\}\,.
\end{align*}
One of classical examples of $E$ with $\kappa(E)<1$ is $E=\left\{1+n^{-a}\mid n=1,2,\dots\right\}$, $a>0$, which yields $\kappa(E) = (1+a)^{-1}$.
Further examples are introduced in \cite{RoSe2023}.
Since many examples are given as sequences, we present an upper bound of $\kappa(E)$ for a sequential $E$ (Corollary~\ref{240704231}).
To our best knowledge, the upper bound has not been given for general sequences in studies of maximal functions.
We note that for a compact $E$, $\kappa(E)$ equals the Minkowski dimension
\begin{align*}
    \dim_{\mathcal{M}}(E)
    &\coloneqq\inf\left\{a>0:\text{$\exists\,\,C>0$ such that $\forall\,\,0<\delta\leq 1$, $N(E,\delta)\leq C\delta^{-a}$}\right\}\,.
\end{align*}

Later, Duoandikoetxea--Vargas \cite{DuoVar1998} extended studies of \cite{SWW1995} to compactly supported Borel measures and Fourier multipliers.
Precisely, they proved the following theorem:
\begin{customthm}{B}\label{thm_DuoVar}\,

    \begin{enumerate}
    \item 
    Let $m$ be the Fourier transform of a compactly supported finite Borel measure and $E\subset(0,\infty)$.
    Assume that 
    $$
        |m(\xi)|\leq C|\xi|^{-a}\quad\text{for some $a>\frac{\kappa(E)}{2}$}.
    $$
    Then $\mathcal{M}_m^E$ is bounded in $L^p(\bR^d)$ for $p>1+\frac{\kappa(E)}{2a}$.
    \item
    Let $m\in C^{s+1}(\bR^d)$ where $s$ is the smallest integer strictly greater than $\frac{d}{2}$, and let $E\subset(0,\infty)$.
    Assume that 
    $$
        \left|\partial^\gamma m(\xi)\right| \leq C|\xi|^{-a}\quad\text{for all $|\gamma|\leq s+1$ and some $a>\kappa(E)/2$}.
    $$
    Then $\mathcal{M}_m^E$ is bounded in $L^p(\bR^d)$ for
    \begin{align}\label{240516302}
        \frac{2d}{d+2a-\kappa(E)}<p<2\frac{d-\kappa(E)}{d-2a}.
    \end{align}
    \end{enumerate}
\end{customthm}
\noindent In case of the spherical average as in \cite{SWW1995}, the decay factor of Theorem~\ref{thm_DuoVar} is $a=\frac{d-1}{2}$.
Thus, one can check that $(1)$ of Theorem~\ref{thm_DuoVar} recovers the result of \cite{SWW1995}.

We pursue further study on $\mathcal{M}_m^E$ in view of $(2)$ of Theorem~\ref{thm_DuoVar}.
In contrast to earlier works relying on the covering property of $E$, our approach to the Minkowski dimension is based on integral quantities (see \eqref{240409516}), motivated by the relation between the Assouad and the Aikawa dimensions. 
This framework enables us to construct a refined square function to control  $\mathcal{M}_m^E$.
Then, following the approach in \cite{LeeSeo_2023}, we make use of fractional calculus, bi-linear interpolation, and $\Sigma^2$ spaces introduced in \cite{LeeSeo_2023}; some of these materials will be presented in Section~\ref{preliminaries}.
These ingredients allow us to improve the regularity condition for $m$ to $s>d\left|\frac{1}{2}-\frac{1}{p}\right|+\kappa(E)\min\left\{\frac{1}{2}, \frac{1}{p}\right\}$.
Moreover, we provide a general result involving Fourier multipliers of limited decay, singular Fourier multipliers, and multipliers with slow decay;
we say that a Fourier multiplier $m$ is with slow decay if there exist $\beta>0$ and $0<\delta<1$ such that
\begin{align}\label{slowdecay}
    \left|\partial^\gamma m(\xi)\right| \lesssim |\xi|^{-\beta-\delta|\gamma|}\quad \text{for all multi-indices $\gamma$.}
\end{align}
One can understand $m$ satisfying \eqref{slowdecay} as a simple analogue of the H\"ormander symbol class.

We now introduce our main result.
Take $\phi \in \mathscr{S}(\R^d)$ such that $\wh{\phi} \equiv 1$ on $B(0,1)$ and $\wh{\phi} \equiv 0$ on $B(0,2)^c$.
Define $\psi$ and $\psi_j$ such that
\begin{align}\label{psi}
\wh{\psi}_j := \wh{\psi}(\,\cdot\,/2^j) =\wh{\phi}(\,\cdot\,/2^j) - \wh{\phi}(\,\cdot\,/2^{j-1})\q \text{for}\q j\in\Z\,.
\end{align}
For a Banach space $\cB$ continuously embedded in $\cD'(\bR^d)$,
we denote $\Sigma^2(\cB)$ the set of all $m\in\cD'(\bR^d\setminus\{0\})$ such that
\begin{align}\label{vv Banach}
\|m\|_{\Sigma^2(\cB)}^2:=\sum_{j\in\Z} \| m(2^j \,\cdot\,) \wh{\psi}(\,\cdot\,) \|_{\cB}^2 <\infty\,.
\end{align}
\begin{theorem}\label{thm_main}
    Let $p\in(1,\infty)$ and $E\subset(0,\infty)$ satisfy $\kappa(E)<1$.
    Suppose that $m$ is of class $\Sigma^2(B_{p_0}^s)$ with $\frac{1}{p_0}=\left|\frac{1}{2}-\frac{1}{p}\right|$ and 
    \begin{align}\label{240328_1552}
        s>\frac{d}{p_0} + \frac{\kappa(E)}{2}\,.
    \end{align}
    Then $\mathcal{M}_m^E$ is bounded on $L^p(\bR^d)$.
\end{theorem}

\begin{remark}\label{rmk_main}
    We note that the endpoint result of Theorem~\ref{thm_CGHS} holds for arbitrary $E\subset\bR$ (See \cite[Corollary 4.2 and (4.16)]{Ch_Gr_Hon_See2005}).
    Thus by interpolating Theorem~\ref{thm_main} (for $p=2$) and the result of Theorem~\ref{thm_CGHS} (for $p=\infty$), the condition \eqref{240328_1552} is replaced by
    \begin{align*}
        s>\frac{d}{p_0} + \kappa(E)\min\left\{\frac{1}{2}, \frac{1}{p}\right\}\,.
    \end{align*}
\end{remark}

We emphasize that the condition for multipliers $m$ in Theorem~\ref{thm_main} encompasses the limited decay condition, so our result with Remark~\ref{rmk_main} recovers $(2)$ of Theorem~\ref{thm_DuoVar} (see examples below).  
On the other hand, one can see from $(1)$ of Theorem~\ref{thm_DuoVar} that the case of $m$ given by the Fourier transform of a measure yields larger range of $p$ for $L^p$ bounds on $\mathcal{M}_m$.
This phenomenon is related to geometric aspects of the measure.
We refer results of Anderson--Hughes--Roos--Seeger \cite{AHRS2021} and Roos--Seeger \cite{RoSe2023}, which concern $L^p$-$L^q$ estimates ($L^p$ improving) of the sphercial maximal functions with $E\subset[1,2]$.
They proved that the sharp $L^p$-$L^q$ bounds is closely related to the Minkowski and the Assouad dimensions of $E$.
We end this section with an example which is not given by the Fourier transform of a compactly supported finite Borel measure.

\begin{example}\,

\begin{enumerate}
    \item Let a multiplier $m$ and a set $E\subset (0,\infty)$ be given as in Theorem~\ref{thm_DuoVar} $(2)$.
    For $\phi$ in the above of \eqref{psi}, denote $m_1:=m\widehat{\phi}$ and $m_2:=m-m_1$.
    Then, since $m_1\in C^{[d/2]+2}(\bR^d)$ is supported on a compact set, $\cM_{m_1}^{\bR_+}$ is bounded on $L^p(\bR^d)$ for $p\in(1,\infty)$, hence so is $\cM_{m_1}^{E}$ (see \cite[Corollary 1.2]{LeeSeo_2023}).
    For $m_2$, note that $|\partial^\alpha m_2(\xi)|\lesssim |\xi|^{-a}$ for all $|\alpha|\leq [d/2]+2$ and $m_2\equiv 0$ on $B(0,1)$, then we obtain that for all $q\in(1,\infty]$ and $r\in (0,[d/2]+2]$,
    \begin{alignat*}{2}
    &\|m_2(2^j\,\cdot\,)\widehat{\psi}\|_{B_{q}^r}\lesssim 2^{j(r-a)}\quad &&\text{if $j\geq 0$}\,;\\
    &\|m_2(2^j\,\cdot\,)\widehat{\psi}\|_{B_{q}^s}=0\quad &&\text{if $j\leq -1$}\,.
    \end{alignat*}
    Consequently, $m_2\in \Sigma(B_{q}^r)$ for $r\in(0,a)$ and $q\in (1,\infty]$, and therefore $\cM_{m_2}^E$ is bounded on $L^p(\bR^d)$ for $p$ in \eqref{240516302} (see Remark~\ref{rmk_main}).

    \item For constants $0<\alpha<1$ and $\beta>0$, 
    let $\mathfrak{m}_{\alpha,\beta} (\xi):= \mathrm{e}^{2\pi i |\xi|^\alpha} m_\beta(\xi)$, where $m_\beta$ is a multiplier vanishing near the origin and satisfying that
    $$
    \left|\partial^\gamma m_\beta(\xi)\right|\lesssim |\xi|^{-\beta-|\gamma|}\quad\text{for any multi-index $\gamma$}\,.
    $$
    Then $\mathfrak{m}_{\alpha,\beta}$ satisfies $|\partial^\gamma \mathfrak{m}_{\alpha, \beta}(\xi)| \lesssim |\xi|^{-\beta-|\gamma|(1-\alpha)}$, which is neither a Mikhlin-type multiplier nor a multiplier of limited decay.
    We note that $\mathfrak{m}_{\alpha,\beta}$ is a variant of singular Fourier multipliers of \cite{Mi1980}.
    By Theorem~\ref{thm_main} and Remark~\ref{rmk_main}, one can check that
    \begin{align}\label{ineq_240405}
        \|\mathcal{M}_{\mathfrak{m}_{\alpha, \beta}} \|_{L^p(\bR^d) \to L^p(\bR^d)}<\infty,
    \end{align}
    whenever $\frac{d-2\beta/\alpha}{2(d-\kappa(E))}<\frac{1}{p}<\frac{d+2\beta/\alpha - \kappa(E)}{2d}$.

    \item Using \eqref{ineq_240405} and the proof of Proposition~1.5 in \cite{LeeSeo_2023}, one can also obtain the following convergence result: 
    Let $\alpha\in(0,1)$, $\beta\in\big(\kappa(E)/2,1\big)$. 
    Suppose $f\in \dot{L}_{\alpha\beta}^p$ for $\frac{d-2\beta}{2(d-\kappa(E))}<\frac{1}{p}<\frac{d+2\beta-\kappa(E)}{2d}$.
    Then we have
    \begin{align*}
        \left|\mathrm{e}^{-it(-\Delta)^{\alpha/2}}f(x) - f(x)\right|=O(t^\beta),\quad t\to 0\text{ on }E\,,
    \end{align*}
    where $\mathrm{e}^{-it(-\Delta)^{\alpha/2}}f$ denotes a solution to fractional half-wave equations with initial data $f$.
\end{enumerate}
\end{example}
We note that $(3)$ of Examples recovers the result of \cite[Remark 1.3]{CKKL2023} when $\alpha\in(0,1)$ and $d\geq2$.
Indeed, the range of $s$ in the case in \cite[Remark 1.3]{CKKL2023} is $s>\frac{r\alpha}{2(1+r)}$ for $\{t_n\}_{n=1}^\infty\in \ell^{r,\infty}$ (for example, $t_n = n^{-1/r}$).
For $L^2$ bounds, we can translate $t_n$ by $1+t_n$, so denote $E=\{1+t_n\}_{n=1}^\infty$. 
By $(3)$, the $L^2$ bounds hold for $f\in L_s^2(\bR^d)$ with $s>\frac{\alpha \kappa(E)}{2}$.
Since $\kappa(E)\leq\frac{r}{1+r}$ (see, \textit{e.g.}, Corollary~\ref{240704231}), we obtain the desired result.

\subsubsection*{Notation}  Throughout the paper, by $A \lesssim B$ we mean $A \le CB$ for harmless constants $C$. We also use $A \simeq B$ when $A\lesssim B$ and $B\lesssim A$.

\section{Preliminaries}\label{preliminaries}

\subsection{Properties of $\Sigma^2$ spaces}
We begin with definitions and properties of Besov spaces.
Let $\phi$ and $\psi$ be functions defined in \eqref{psi} and its above.
We denote $B_{p}^s:=B_{p}^s(\mathbb R^d):=B_{p,p}^s(\bR^d)$ a Besov space defined by
\[
 B_{p,p}^s(\bR^d):= \Big\{f\in \mathscr S'(\mathbb R^d) : \Big(\| \phi \ast f\|_p^p+ \sum_{j\geq 1} 2^{sjp}\| \psi_j \ast f\|_p^p\Big)^{1/p} <\infty\Big\}
\]
for $1\leq p<\infty$, and 
\[
 B_{\infty,\infty}^s(\bR^d):= \Big\{f\in \mathscr S'(\mathbb R^d) : \| \phi \ast f\|_\infty+ \sup_{j\geq 1} \left(2^{sj}\| \psi_j \ast f\|_\infty\right) <\infty\Big\}\,;
\]
here $\mathscr S'$ means a class of tempered distributions.
For properties of Besov spaces, we mainly recommend \cite[Section 2.3]{Tri1978} (see also \cite{BerLof1976, FJW1956, Gra2014}).
Note that
\begin{align}\label{holder-besov}
\begin{alignedat}{2}
    &L^2_s(\mathbb R^d)=B_2^s(\mathbb R^d)\quad &&\text{for all $s\in\bR$}\\
    &C^{n,s'}(\mathbb R^d)=B_\infty^{n+s'}(\mathbb R^d) \quad&&\text{for all $n\in\bN\cup\{0\}$ and  $s'\in(0,1)$;}
\end{alignedat}
\end{align}
for the last equality, see \cite[Section 2.7]{Tri1978} or \cite[Remark 2.2.2]{Gra2014}.
Here, $L^2_s(\bR^d)$ denotes the Bessel potential space $(1-\Delta)^{-s/2}L^2(\bR^d)$ equipped with the norm
$$
\|f\|_{L^2_s(\mathbb{R}^d)}:=\left(\int_{\mathbb R^d}\left|(1-\Delta)^{s/2}f(x)\right|^2\dd x\right)^{1/2}\simeq \left(\int_{\mathbb{R}^d}\left|(1+|\xi|^2)^{s/2}\widehat{f}(\xi)\right|^2\dd \xi\right)^{1/2}\,,
$$
and 
$C^{n,s'}(\mathbb R^d)$ denotes a H\"older space equipped with the norm
$$
    \|f\|_{C^{n,s'}(\bR^d)}:=\sum_{|\gamma|\leq n}\sup_{x\in\bR^d}|\partial^\gamma f(x)|+\sum_{|\gamma|=n}\sup_{x,y\in\bR^d,x\neq y}\frac{|\partial^\gamma f(x)-\partial^\gamma f(y)|}{|x-y|^{s'}}\,.
$$

We now introduce some useful properties of $\Sigma^2$ spaces defined by \eqref{vv Banach}.
\begin{lemma}\label{231109931}
\,

\begin{enumerate} 
    \item For any $1<p\leq\infty$ and $s\in\bR$, $\Sigma^2(B_p^s)$ is a Banach space. 

    \item For any $l=0,\,1,\,\ldots$,
    $$
    \|m\|_{\Sigma^2(L^2_l)}^2:=\|m\|_{\Sigma^2(B_2^l)}^2\simeq \sum_{|\gamma|\leq l}\int_{\bR^d}|\partial^\gamma m(\xi)|^2|\xi|^{2|\gamma|-d}\dd \xi\,.
    $$
    
    \item 
    For any $1<p\leq\infty$ and $s\in\bR$,
    $$
    \sup_{r>0}\|m(r\,\cdot\,)\|_{\Sigma^2(B^s_p)}\lesssim \|m\|_{\Sigma^2(B^s_p)}\quad\text{and}\quad \|\xi\cdot \nabla m(\xi)\|_{\Sigma^2(B_p^{s})}\lesssim \|m\|_{\Sigma^2(B_p^{s+1})}\,,
    $$
where $\xi\cdot \nabla:=\xi_1\partial_{\xi_1}+\cdots+\xi_d\partial_{\xi_d}$.
    
    \item Let $p_0,\,p_1\in (1,\infty]$, $s_0,\,s_1\in\bR$. 
    For any $0<\theta<1$, 
    $$  \big[\Sigma^2(B_{p_0}^{s_0}),\Sigma^2(B_{p_1}^{s_1})\big]_\theta=\Sigma^2\big(B_{p_\theta}^{s_\theta}\big)\,,
    $$
    where $[X,Y]_\theta$ is an complex interpolation space, $\frac{1}{p_\theta}:=\frac{1-\theta}{p_0}+\frac{\theta}{p_1}$, and $s_\theta:=(1-\theta)s_0+\theta s_1$. 
\end{enumerate}
\end{lemma}

\begin{proof}
(1) and (4) are provided in \cite[Subsection 1.2]{LeeSeo_2023}; note the interpolation result of Besov spaces,
$$
    [B_{p_0}^{s_0}, B_{p_1}^{s_1}]_\theta = B_{p_\theta}^{s_\theta}
$$
(see, \textit{e.g.}, \cite[Theorem~6.4.5, (6)]{BerLof1976}).
(2) is provided in \cite[Proposition 2.2]{Lo1}.

    (3) To prove the first inequality, let $r=2^{j_0}r'$, where $j\in\bZ$ and $r'\in[1,2)$.
    Note that for any $r'\in[1,2)$,
    $$
    \sum_{|n|\leq 2}\wh{\psi}_{n}(r'\,\cdot\,)\equiv 1\quad \text{on}\quad \mathrm{supp}(\wh{\psi})\,,
    $$
    which implies that
    \begin{align*}
        &\|m(r\,\cdot\,)\|_{\Sigma^2(B_p^s)}^2=\sum_{j\in\Z} \big\| m(r2^j \,\cdot\,) \wh{\psi} \big\|_{B_p^s}^2=\sum_{j\in\Z}\bigg\| m(r'2^j \,\cdot\,) \,\sum_{|n|\leq 2}\wh{\psi}_{n}(r'\,\cdot\,)\,\wh{\psi}\bigg\|_{B_p^s}^2\,.
    \end{align*}
From the definition of $B^s_p$, we obtain that
\begin{align*}
\bigg\| m(r'2^j \,\cdot\,) \,\sum_{|n|\leq 2}\wh{\psi}_{n}(r'\,\cdot\,)\,\wh{\psi}\bigg\|_{B_p^s}\lesssim\,& \sum_{|n|\leq 2}\big\| m(r'2^j \,\cdot\,)\wh{\psi}_{n}(r'\,\cdot\,)\big\|_{B_p^s}\\
=\,&\sum_{|n|\leq 2}\big\| m(r'2^{j} \,\cdot\,)\wh{\psi}(r'2^{-n}\,\cdot\,)\big\|_{B_p^s}\\
\lesssim\,&\sum_{|n|\leq 2} \big\| m(2^{j+n} \,\cdot\,)\wh{\psi}\big\|_{B_p^s}
\end{align*}
(note that $1\leq r'<2$).
Therefore we have
\begin{align*}
\sup_{r>0}\|m(r\,\cdot\,)\|_{\Sigma^2(B_p^s)}^2\,&=\sup_{r'\in[1,2)}\sum_{j\in\Z}\bigg\| m(r'2^j \,\cdot\,) \,\sum_{|n|\leq 2}\wh{\psi}_{n}(r'\,\cdot\,)\,\wh{\psi}\bigg\|_{B_p^s}^2\\
&\lesssim \sum_{j\in\bZ}\big\| m(2^j \,\cdot\,)\wh{\psi}\big\|_{B_p^s}^2\,.
\end{align*}

Next, we prove the second inequality.
For fixed $k=1,\,\ldots,\,d$, we have 
\begin{align}\label{240214634}
&\big\|\big(\xi_k \partial_{\xi_k} m\big)(2^j\,\cdot\,)\wh{\psi}\big\|_{B_p^s}\nonumber\\
=\,&\big\|\partial_{\xi_k}\big( m(2^j\,\cdot\,)\big) \cdot\xi_k\wh{\psi}\big\|_{B_p^s}\\
\leq\, &\big\|\partial_{\xi_k}\big(m(2^j\,\cdot\,) \xi_k\wh{\psi}\,\big)\big\|_{B_p^{s}}+\big\|m(2^j\,\cdot\,) \partial_{\xi_k}\big(\xi_k\wh{\psi}\,\big)\big\|_{B_p^s} \nonumber\\
\lesssim \,&\|m(2^j\,\cdot\,) \xi_k\wh{\psi}\|_{B_p^{s+1}}+
\|m(2^j\,\cdot\,) \partial_{\xi_k}\big(\xi_k\wh{\psi}\,\big)\|_{B_p^{s+1}}\,.\nonumber
\end{align}
Since $\xi_k\wh{\psi}(\xi)$ is smooth and $\sum_{|n|\leq 1}\wh{\psi}_{n}\equiv 1$ on $\mathrm{supp}\, \wh{\psi}$, we obtain that
\begin{align}\label{240214635}
\begin{split}
&\|m(2^j\,\cdot\,) \xi_k\wh{\psi}\|_{B_p^{s+1}}+
\|m(2^j\,\cdot\,) \partial_{\xi_k} \big(\xi_k\wh{\psi}\big)\|_{B_p^{s+1}} \\
\lesssim \,&\sum_{|n|\leq 1}\|m(2^j\,\cdot\,) \wh{\psi}_n\|_{B_p^{s+1}}\simeq \sum_{|n|\leq 1}\|m(2^{j+n}\,\cdot\,) \wh{\psi}\|_{B_p^{s+1}}\,.
\end{split}
\end{align}
\eqref{240214634} and \eqref{240214635} imply
\begin{align*}
\|\xi_k\partial_{\xi_k} m(\xi)\|_{\Sigma^2(B_p^{s})}^2&=\sum_{j\in\bZ}\|\big(\xi_k\partial_{\xi_k} m\big)(2^j\,\cdot\,)\wh{\psi}\|_{B_p^{s}}^2\\
&\lesssim \sum_{j\in\bZ}\|m(2^j\,\cdot\,)\wh{\psi}\|_{B_p^{s+1}}^2\\
&=\|m\|_{\Sigma^2(B_p^{s+1})}^2\,.
\end{align*}
Therefore the proof is completed.
\end{proof}

\subsection{Fractional calculus}
We define the Riemann--Liouville integrals and fractional derivatives of order $\alpha \in (0,1)$. 
For $t>0$,
\begin{align*}
  I_{0+}^\alpha f(t) := \frac{1}{\Gamma(\alpha)} \int_0^t (t-s)^{\alpha-1} f(s) ds\quad,\quad  D_{0+}^{\alpha}F(t) := \frac{d}{dt}\big(I_{0+}^{1-\alpha}F\big)(t)
\end{align*}
where $f$ is a locally integrable function on $[0,\infty)$ and $F$ satisfies $I_{0+}^{1-\alpha}F$ is absolutely continuous.
For further use of $I_{0+}^\alpha$ and $D_{0+}^\alpha$, we need the following lemma:
\begin{lemma}[Lemma~2.1, \cite{LeeSeo_2023}]\label{lem_fraccal}
Let $F\in C_{loc}\big([0,\infty)\big)\cap C^{0,\beta}_{loc}\big((0,\infty)\big)$ for a fixed $\beta\in(0,1]$.
Then for any $\alpha\in(0,\beta)$, $D^{\alpha}_{0+}F$ satisfies the following identity:
\begin{align}\label{220514332}
D^{\alpha}_{0+}F(t)=\frac{1}{\Gamma(1-\alpha)}\left(\frac{F(t)}{t^{\alpha}}+\alpha\int_0^t\frac{F(t)-F(s)}{(t-s)^{1+\alpha}}ds\right)\,.
\end{align}
Moreover, we have
\begin{align}\label{220514330}
F(t)=I_{0+}^{\alpha}D_{0+}^{\alpha}F(t)\,.
\end{align}
\end{lemma}
It should be noted that $\frac{t^\alpha}{\Gamma(\alpha)}F(0)$ is added in the right-hand side of \eqref{220514330} in Lemma~2.1 of \cite{LeeSeo_2023}, which is an error and should be replaced by Lemma~\ref{lem_fraccal}. 
We also note that arguments in \cite{LeeSeo_2023} are valid regardless of the error.
For further discussion on fractional calculus, we refer to \cite{Kip1960,Sam_Kil_Ma1993}.

\section{Proof of Theorem~\ref{thm_main}}

\subsection{Useful lemmas}

In this subsection, we present some lemmas that are useful in our proof and examples.
The first lemma provides a refined square function argument for maximal operators on fractal sets in $\bR_+$.
\begin{lemma}\label{lem_sqrftn}
Let $E\subset\bR_+$ be a set such that $\{2^j:j\in\bZ\}\subset E$, $\alpha$ and $\beta$ be positive constants such that $0<\beta<\alpha\leq \frac{1}{2}$, and $F\in C_{loc}\big([0,\infty)\big)\cap C^{0,\alpha+\varepsilon}_{loc}\big((0,\infty)\big)$ for some $\varepsilon>0$.
Then we have
    \begin{align*}
        \sup_{t\in E}|F(t)|^2\lesssim_{\alpha,\beta} \sum_{j\in\bZ}\int_1^2 d(s,\widetilde{E}_j)^{-1+2\beta}|D^\alpha_{0+}F_j(s)|^2\dd s\,,
    \end{align*}
    where $F_j(s):=F(2^js)$, $\widetilde{E}_j:=\big(2^{-j}E\big)\cap [1,2]$, and $d(s,A):=\inf_{r\in A}|s-r|$ for a set $A\subset \bR_+$.
\end{lemma}

\begin{proof}
By Lemma~\ref{lem_fraccal}, we obtain that for any $t\in E$,
\begin{align}
|F(t)|^2\leq &\left(\int_0^t(t-s)^{-1+\alpha}\left|D^\alpha_{0+}F(s)\right|\dd s\right)^2\nonumber\\
\leq &\int_0^t(t-s)^{-1+2\beta}s^{2(\alpha-\beta)}\left|D^\alpha_{0+}F(s)\right|^2\dd s\label{2404171134}\\
\lesssim &\int_0^\infty d(s,E)^{-1+2\beta}s^{2(\alpha-\beta)}\left|D^\alpha_{0+}F(s)\right|^2\dd s\nonumber\\
= &\sum_{j\in\bZ}\int_1^2 d(s,E_j)^{-1+2\beta}s^{2
(\alpha-\beta)}\left|2^{j\alpha}\big(D^\alpha_{0+}F\big)(2^js)\right|^2\dd s\,.\label{240409528}
\end{align}
Here, \eqref{2404171134} is implied by that because $0<2(\alpha-\beta)<1$,
$$
\int_0^t(t-s)^{-1+2(\alpha-\beta)}s^{-2(\alpha-\beta)}\dd s=N(\alpha,\beta)<\infty\,,
$$
and \eqref{240409528} follows from that $\{2^j\,:\,j\in\bZ\}\subset E$.
One can observe that
\begin{align*}
2^{j\alpha}\left(D_{0+}^\alpha F\right)(2^js)\,&=\frac{1}{\Gamma(1-\alpha)}\left(\frac{F(2^js)}{s^{\alpha}}+\alpha2^{j\alpha}\int_0^{2^js}\frac{F(2^js)-F(r)}{(2^js-r)^{1+\alpha}}\dd r\right)\\
&=\frac{1}{\Gamma(1-\alpha)}\left(\frac{F(2^js)}{s^{\alpha}}+\alpha\int_0^s\frac{F(2^js)-F(2^jr)}{(s-r)^{1+\alpha}}\dd r\right)\\
&=D_{0+}^\alpha F_j(s)\,,
\end{align*}
where $F_j(s)=F(2^js)$.
Since the factor of $s^{2(\alpha-\beta)}$ in \eqref{240409528} disappears simply because $s\in[1.2]$, we obtain that for $t\in E$,
\begin{align*}
|F(t)|^2\lesssim \sum_{j\in\bZ}\int_1^2 d(s,E_j)^{-1+2\beta}\left|D^\alpha_{0+}F_j(s)\right|^2\dd s\,.
\end{align*}
\end{proof}

The second lemma provides an equivalent definition of the Minkowski dimension of compact sets.
This lemma is motivated by the relation between the Assouad dimension and the Aikawa dimension, introduced in \cite{LeTu20213}.

\begin{lemma}\label{231109742}
    Let $E$ be a non-empty set in $[1,2]$.
    For any $a\in(0,1)$, 
    \begin{align}\label{240313416}
        \sup_{0<\delta\leq 1}\delta^a N(E,\delta)\lesssim_a \int_1^2d(t,E)^{-1+a}\dd t\lesssim_a 1+\int_0^1 \lambda^{a}N(E,\lambda)\frac{\dd \lambda}{\lambda}\,.
    \end{align}
    In particular, the Minkowski dimension of $E$ equals
    \begin{align}\label{240409516}
        \inf\Big\{b>0\,:\,\int_1^2d(t,E)^{-1+b}\dd t<\infty\Big\}\,.
    \end{align}
\end{lemma}

\begin{proof}
Observe that
\begin{align}\label{231109722}
    \int_1^2d(t,E)^{-1+a}\dd t&=(1-a) \int_0^\infty \lambda^{-2+a}\big|\{t\in[1,2]\,:\,d(t,E)<\lambda\}\big|\dd \lambda\,.
\end{align}
For a fixed $\delta\in(0,1]$, we denote 
$$
\cI_\delta:=\big\{[k\delta,(k+1)\delta]\,:\,[k\delta,(k+1)\delta]\cap E\neq \emptyset\big\}\,.
$$
One can observe that if $I\in\cI_\delta$, then $d(\cdot,E)\leq \delta$ on $I$.
This implies that
\begin{align}\label{240422418}
&\delta N(E,\delta)=\bigg|\bigcup_{I\in \cI_\delta}I\bigg| \nonumber\\
\leq \,\,&\big|\{t>0\,:\,d(t,E)\leq \delta\}\big|\\
\leq \,\,&
(1-a)\delta^{1-a}\int_\delta^\infty \lambda^{-2+a}\big|\{t\in [1,2]\,:\,d(t,E)<\lambda\}\big|\dd \lambda+2\delta\nonumber\\
\leq \,\,&3(1-a)\delta^{1-a}\int_0^\infty \lambda^{-2+a}\big|\{t\in [1,2]\,:\,d(t,E)<\lambda\}\big|\dd \lambda\,.\nonumber
\end{align}
Here, the last inequality follows from that because 
$$
|\{t\in [1,2]\,:\,d(t,E)<\lambda\}|=1\quad \text{for all}\quad \lambda\geq 1\,,
$$
we obtain that for any $0<\delta\leq 1$,
$$
\delta\leq (1-a)\delta^{1-a} \int_1^\infty \lambda^{-2+a}|\{t\in [1,2]\,:\,d(t,E)<\lambda\}|\dd \lambda\,.
$$
Combining \eqref{240422418} with \eqref{231109722}, the first inequality in \eqref{240313416} is proved.

To prove the second inequality in \eqref{240313416}, 
observe that for a fixed $\lambda>0$,
$$
\big\{t\in[1,2]\,:\,d(t,E)<\lambda\big\}\subset \bigcup_{I\in\cI_\lambda}\widetilde{I}\,,
$$
where $\widetilde{I}$ is an interval with the same center as $I$ and $\ell(\widetilde{I})=3\ell(I)$.
Due to the definition of $N(E,\delta)$, we have
$$
\big|\big\{t\in[1,2]\,:\,d(t,E)<\lambda\big\}\big|\leq \bigg|\bigcup_{I\in \cI_\lambda}\widetilde{I}\,\bigg|\leq 3\lambda N(E,\lambda)\,.
$$
By combining this with \eqref{231109722}, we obtain that
\begin{align*}
&\int_0^\infty \lambda^{-2+a}|\{t\in[1,2]\,:\,d(t,E)<\lambda\}|\dd \lambda\leq \frac{1}{1-a}+3\int_0^1\lambda^{-1+a}N(E,\lambda)\dd \lambda\,.
\end{align*}
Therefore the proof of \eqref{240313416} is completed, and
it directly follows from \eqref{240313416} that $\dim_\cM E$ equals \eqref{240409516}.
\end{proof}

We note that Lemma~\ref{231109742} still true for $E \subset [a,b]$ for arbitrary $a,b$.
Exactly same argument works for the case of $[a,b]$.

It is known that the sequence $\{1+n^{-1/r}\}_{n\in \bN}$ is of the Minkowski dimension $\frac{r}{1+r}$ (see, \textit{e.g.}, \cite[Examples]{SWW1995}).
We provide similar results for bounded decreasing sequences by using Lemma~\ref{231109742}.
As applications of Lemma~\ref{231109742}, we provide similar results for bounded decreasing sequences, which are not used in the proof of the main result.

\begin{corollary}\label{240704231}\,
Let $(t_n)_{n\in\bN}$ be a decreasing sequence with $t_n\rightarrow 0$.
\begin{enumerate}
    \item  $\dim_{\cM}\big(\{t_n\}_{n\in\bN}\big)$ equals the infimum of all $a>0$ such that
    $$
    \sum_{n=1}^\infty(t_n-t_{n+1})^a<\infty\,.
    $$

    \item Let $(t_n)_{n\in\bN}\in \ell^{r,\infty}(\bN)$ for $0<r<\infty$, \textit{i.e.}, there exists $N>0$ such that for any $\delta>0$, 
\begin{align}\label{240704226}
\#\{n\,:\,t_n\geq \delta\}\leq N\delta^{-r}\,.
\end{align}
Then $\dim_{\cM}\big(\{t_n\}_{n\in\bN}\big)\leq \frac{r}{1+r}$.
In particular, for any real valued $(a_n)_{n\in\bN}\in\ell^{r,\infty}(\bN)$, $\dim_{\cM}\big(\{a_n\}_{n\in\bN}\big)\leq \frac{r}{1+r}$\,.

    \item If we additionally assume that $(t_n-t_{n+1})_{n\in\bN}$ is also decreasing, then $\dim_{\cM}\big(\{t_n\}_{n\in\bN}\big)=\frac{r_0}{1+r_0}$, where $r_0$ is the infimum of all $r>0$ such that $(t_n)_{n\in\bN}\in \ell^{r,\infty}(\bN)$.
\end{enumerate}
\end{corollary}

\begin{proof}
(1) By translation and dilation, we assume that $t_1=1$ and $t_n\rightarrow 0$.
    Observe that
    \begin{align}
        \int_0^1 d\big(s,\{t_n\}\big)^{-1+a}\dd s = \,&\sum_{n=1}^\infty\int_{t_{n+1}}^{t_n}\max\big(|t_n-s|,|s-t_{n+1}|\big)^{-1+a}\dd s\nonumber\\
        \simeq \,&\sum_{n=1}^\infty(t_n-t_{n+1})^a\,.\label{240704244}
    \end{align}
    By Lemma~\ref{231109742} associated with the interval $[0,1]$, the proof is completed.

(2)
By rearranging, we assume that $(t_n)$ is strictly decreasing and converges to $0$.
    We denote $N_{(k)}:=\max\{n\,:\,t_n\geq 2^{-k}\}$.
    Then
    \begin{align*}
    \sum_{n=N_{(1)}+1}^\infty(t_{n}-t_{n+1})^a\,&=\sum_{k=1}^\infty\sum_{n=N_{(k)}+1}^{N_{(k+1)}}(t_n-t_{n+1})^a\\
    &\leq \sum_{k=1}^\infty\big(N_{(k+1)}-N_{(k)}\big)^{1-a}(t_{N_{(k)}+1})^a\leq \sum_{k=1}^\infty 2^{-ak}N_{(k+1)}^{1-a}\,,
    \end{align*}
    where the first inequality follows from the H\"older inequality.
It follows from \eqref{240704226} that $N_{(k)}\lesssim 2^{-kr}$.
Therefore the last term is finite if $a>\frac{r}{1+r}$.
Using Lemma~\ref{231109742} and \eqref{240704244}, we conclude that $\dim_{\cM}\big(\{t_n\}\big)\leq \frac{r}{1+r}$.

(3) By (2) of this lemma, we only need to prove that  $\dim_{\cM}\big(\{t_n\}\big)\geq \frac{r_0}{1+r_0}$.
Therefore we only prove that if $\dim_{\cM}\big(\{t_n\}\big)<\alpha<1$, then $(t_n)\in\ell^{\alpha/(1-\alpha),\infty}(\bN)$.
We denote $N_\delta:=\max\big\{n\in\bN\,:\,t_n\geq \delta\big\}$.
Since $t_n\rightarrow 0$ and $(t_n-t_{n+1})$ is decreasing, we have
\begin{align}
\delta^{\alpha/(1-\alpha)}N_\delta \,&\leq \sum_{n=1}^{N_\delta}t_n^{\alpha/(1-\alpha)}=\sum_{n=1}^{N_\delta}\bigg[\sum_{k=n}^\infty \big(t_k-t_{k+1}\big)\bigg]^{\alpha/(1-\alpha)}\nonumber\\
&\leq \sum_{n=1}^{\infty}\bigg[(t_n-t_{n+1})^{1-\alpha}\sum_{k=n}^\infty \big(t_k-t_{k+1}\big)^\alpha\bigg]^{\alpha/(1-\alpha)}\label{240704851}\\
&\leq \bigg[\sum_{k=1}^\infty \big(t_k-t_{k+1}\big)^\alpha\bigg]^{1/(1-\alpha)}\nonumber
\end{align}
Due to (1) of this lemma and that $\alpha>\dim_{\cM}\big(\{t_n\}\big)$, we obtain that the last term in \eqref{240704851} is finite.
This implies \eqref{240704226} with $r=\frac{\alpha}{1-\alpha}$.
\end{proof}

\subsection{Proof of Theorem~\ref{thm_main}}

We first note that for $E\subset \bR_+$ and 
$\widetilde{E}:=E\cup \{2^j\,;\,j\in\bZ\}$,
we have $\kappa(E)=\kappa\big(\widetilde{E}\big)$ and 
$$
\sup_{E}\left|T_{m(t\,\cdot\,)}f\right|\leq \sup_{\widetilde{E}}\left|T_{m(t\,\cdot\,)}f\right|\,.
$$
This implies that we only need to prove for the case of when
\begin{align}\label{240422520}
\{2^j\,:\,j\in\bZ\}\subset E\,.
\end{align}
Therefore we additionally assume \eqref{240422520}.

Noting that $\frac{d}{p_0}+\frac{\kappa(E)}{2}:=d\big|\frac{1}{2}-\frac{1}{p}\big|+\frac{\kappa(E)}{2}<s$, 
take $\alpha\in \left(\frac{\kappa(E)}{2},\frac{1}{2}\right)$ such that 
$$
\frac{d}{p_0}+\alpha<s\,.
$$
We assume that $m\in \Sigma^2(B_{p_0}^s)$.
There is an $\varepsilon_0>0$ such that $s \geq d/p_0 + \alpha +\varepsilon_0$ and $0<\alpha+\varepsilon_0<1$. 
Note from \cite[Theorem~2.8.1 (c)]{Tri1978} that $B_{p_0}^s(\bR^d)\subset B_\infty^{\alpha+\varepsilon_0}(\bR^d)=C^{0,\alpha+\varepsilon_0}(\bR^d)$.
Thus, one has
\begin{align}\label{2404171130}
\sum_{j\in\bZ}\|m(2^j\,\cdot\,)\widehat{\psi}(\,\cdot\,)\|_{C^{0,\alpha+\varepsilon_0}}^2<\infty\,.
\end{align}
This implies that $m\in C(\bR^d)$, $m(0)=0$, and
$$
|m(\xi)-m(\xi')|\lesssim \left(\frac{|\xi-\xi'|}{|\xi|}\right)^{\alpha+\varepsilon_0}
$$
for all $\xi,\xi'\in\bR^d$ with $|\xi-\xi'|\leq \frac{|\xi|}{2}$.
Therefore, one can observe that for any Schwartz function $f$ and $x\in\bR^d$, $F(t):=T_{m(t\,\cdot\,)}f(x)$ satisfies $F\in C_{loc}\big([0,\infty)\big)\cap C^{0,\alpha+\varepsilon_0}_{loc}\big((0,\infty)\big)$.

Take any $\beta\in \big(\frac{\kappa(E)}{2},\alpha\big)$, and apply Lemma~\ref{lem_sqrftn} for $F(t):=T_{m(t\,\cdot\,)}f(x)$ to obtain
\begin{align}\label{2404171115}
\sup_{t\in E}\left|T_{m(t\,\cdot\,)}f\right|^2&\lesssim \sum_{j\in\bZ}\int_1^2 d(s,\widetilde{E}_j)^{-1+2\beta}\left|D^\alpha_{0+}\left(T_{m(2^js\,\cdot\,)}f\right)\right|^2\dd s\,.
\end{align}
We apply \eqref{220514332} of Lemma~\ref{lem_fraccal} to obtain that for any $s\in[1,2]$, 
\begin{align}\label{2404171116}
\left|D^\alpha_{0+}\left(T_{m(2^js\,\cdot\,)}f\right)\right|\lesssim \left|T_{m(2^js\,\cdot\,)}f\right|+\left|T_{\wt{m}(2^js\,\cdot\,)}f\right|\,,
\end{align}
where $D^\alpha_{0+}$ is for the parameter $s$ and 
\begin{align}\label{2508291038}
\wt{m}(\xi):=\int_0^1\frac{m(\xi)-m(r\xi)}{(1-r)^{1+\alpha}}\dd r
\end{align}
(due to \eqref{2404171130}, $\wt{m}$ is well-defined and it is bounded).
Then by \eqref{2404171115} and \eqref{2404171116}, we have
\begin{align}\label{240422402}
\sup_{t\in E}\left|T_{m(t\,\cdot\,)}f\right|\lesssim \|G(\wt{m},f)\|_{\cH}+\|G(m,f)\|_{\cH}\,,
\end{align}
where $G(m,f):=\{T_{m(2^js\,\cdot\,)}f\}_{j\in\bZ,s\in [1,2]}$ and $\cH$ is an $L^2$-space defined by the following norm: 
$$
\big\|\big\{F(j,s)\big\}_{j\in\bZ,s\in[1,2]}\big\|_{\cH}^2:=\sum_{j\in\bZ}\int_1^2d(s,\widetilde{E}_j)^{-1+2\beta}|F(j,s)|^2\dd s\,.
$$
Note that $G(m,f)$ can be interpreted by a $\mathcal{H}$-valued Fourier multiplier operator.
Let $M_m:\bR^d\rightarrow \cH$ be given by
$$
M_m(\xi)=\{m(2^js\xi)\}_{j\in\bZ,s\in[1,2]}\,,
$$
so that $G(m,f)=T_{M_m}f:\bR^d\rightarrow \cH$.
It is not difficult to show that
\begin{align}\label{231228_1628}
    \| G(m,f) \|_{L^2(\bR^d;\mathcal{H})} 
    \lesssim 
    \|M_m\|_{L^\infty(\bR^d, \mathcal{H})} \|f\|_{L^2(\bR^d)}.
\end{align}
Moreover, the vector-valued Carder\'on-Zygmund theory (see \cite{Ru_Ru_To1986}) implies that $G(m,f)$ satisfies
\begin{align*}
\begin{split}
    \| G(m,f) \|_{L^1(\bR^d;\mathcal{H})}&\lesssim \sup_{k\in \bZ} \| M_m(2^k\,\cdot\,)\widehat{\psi}\|_{L_{d/2+\varepsilon}^2(\bR^d; \mathcal{H})} \|f\|_{H^1(\bR^d)}\,,\\
    \| G(m,f) \|_{BMO(\bR^d;\mathcal{H})}&\lesssim \sup_{k\in \bZ} \| M_m(2^k\,\cdot\,)\widehat{\psi}\|_{L_{d/2+\varepsilon}^2(\bR^d; \mathcal{H})}\|f\|_{L^\infty(\bR^d)}\,.
\end{split}
\end{align*}
for any $\varepsilon>0$.

We temporarily assume
\begin{gather}
        \|M_m\|_{L^\infty(\bR^d; \mathcal{H})}
    \lesssim \left(\sum_{j\in\bZ}\|m(2^j\,\cdot\,)\widehat{\psi}\|_\infty^2\right)^{1/2}
    = \|m\|_{\Sigma^2(L^\infty)}\,,\label{ineq_L2esti}\\
    \sup_{k\in\bZ}\|M_m(2^k\,\cdot\,)\widehat{\psi}\|_{L_\beta^2(\bR^d; \mathcal{H})} \simeq \|m\|_{\Sigma^2(L_\beta^2)}\quad\text{for all}\quad \beta\geq0\,,\label{ineq_endpts}
\end{gather}
and the following embedding result for $\widetilde{m}$:
\begin{lemma}\label{lem_emb}
Let $\alpha\in(0,1)$ and let $\widetilde{m}$ be defined in \eqref{2508291038}.
For any $1<p\leq \infty$, $s\in\bR$, and $\varepsilon>0$, 
$$
\|\wt{m}\|_{\Sigma^2(B_p^s)}\leq C(d,s,p,\alpha,\varepsilon)\|m\|_{\Sigma^2(B_p^{s+\alpha+\varepsilon})}\,.
$$
\end{lemma}
Since it is clear from \eqref{holder-besov} and the definition of $\Sigma^2$ spaces that
$$
\Sigma^2(L^{\infty})\supset \Sigma^2(B_\infty^\varepsilon)\supset \Sigma^2(B_\infty^{\alpha+2\varepsilon})\quad \text{and}\quad \Sigma^2(L^2_{d/2+\varepsilon})\supset \Sigma^2(L^2_{d/2+\alpha+2\varepsilon})\,,
$$
\eqref{231228_1628} - \eqref{ineq_endpts} and Lemma~\ref{lem_emb} imply that for any $\varepsilon>0$,
    \begin{align}
    \| G(m, f) \|_{L^2(\bR^d; \mathcal{H})}+\| G(\widetilde{m}, f) \|_{L^2(\bR^d; \mathcal{H})}
    &\lesssim \|m\|_{\Sigma^2(B_\infty^{\alpha+2\varepsilon})} \|f\|_{L^2(\bR^d)}\,,\label{231228_1707}\\
    \| G(m, f) \|_{L^1(\bR^d; \mathcal{H})}+\| G(\widetilde{m}, f) \|_{L^1(\bR^d; \mathcal{H})}
    &\lesssim \|m\|_{\Sigma^2(L^2_{d/2+\alpha+2\varepsilon})}\|f\|_{H^1(\bR^d)}\,,\label{231228_1708}\\
    \| G(m, f) \|_{BMO(\bR^d; \mathcal{H})}+\| G(\widetilde{m}, f) \|_{BMO(\bR^d; \mathcal{H})}
    &\lesssim \|m\|_{\Sigma^2(L^2_{d/2+\alpha+2\varepsilon})} \|f\|_{L^\infty(\bR^d)}\,.\label{231228_1709}
\end{align}

Consider $G(m,f)$ as a bi-linear operator and apply the following lemma to interpolate the pairs 
\{\eqref{231228_1707}, \eqref{231228_1708}\} and \{\eqref{231228_1707}, \eqref{231228_1709}\}:
\begin{lemma}[paragraph 10.1, \cite{Cal1964}]\label{complex interpolation}
    Let $L(x_1, \dots, x_n)$ be a multilinear mapping defined for $x_i \in A_i \cap B_i, i=1,\dots, n$ with values in $A\cap B$ and such that
  \begin{align*}
    \big\| L(x_1, x_2, \dots, x_n) \big\|_A \leq M_0 \prod_{i=1}^n \| x_i \|_{A_i},\\
    \big\| L(x_1, x_2, \dots, x_n) \big\|_B \leq M_1 \prod_{i=1}^n \| x_i \|_{B_i}\,.
  \end{align*}
  Then we have
  \begin{align*}
    \big\| L(x_1,x_2, \dots, x_n) \big\|_C \leq M_0^{1-\theta}M_1^\theta  \prod_{i=1}^n \| x_i \|_{C_i},
  \end{align*}
  where $C=[A,B]_{\theta}$, $C_i = [A_i, B_i]_\theta$, $i=1,\dots,n$
and thus $L$ can be extended continuously to a multilinear mapping of $C_1\times \ldots\times C_n$ into $C$\,.
\end{lemma}

\noindent For $\frac{1}{p_0} := \left|\frac{1}{2}-\frac{1}{p}\right|$ and $\theta:=\frac{2}{p_0}$, we have  
\begin{align*}
    &\left[\Sigma^2\left(C^{0, \alpha+2\varepsilon}\right), \Sigma^2\left(L^2_{d/2+\alpha+2\varepsilon}\right)\right]_\theta
    \\
    =\,& \left[\Sigma^2\left(B^{\alpha+2\varepsilon}_{\infty}\right), \Sigma^2\left(B_2^{d/2+\alpha+2\varepsilon}\right)\right]_\theta=\Sigma^2\left( B_{p_0}^{d/p_0 + \alpha+2\varepsilon}\right)
\end{align*}
(see \eqref{holder-besov} and Lemma~\ref{231109931}).
Therefore, from \eqref{231228_1707} - \eqref{231228_1709} and Lemma~\ref{complex interpolation},  we have
$$
\big\|G(m,f)\big\|_{L^p(\bR^d;\cH)}+\big\|G(\widetilde{m},f)\big\|_{L^p(\bR^d;\cH)}\lesssim \|m\|_{\Sigma^2\left(B_{p_0}^{d/p_0+\alpha+2\varepsilon}\right)}\|f\|_p\,.
$$
Combining this with \eqref{240422402}, we conclude that 
\begin{align*}
\Big\|\sup_{t\in E}\left|T_{m(t\,\cdot\,)}f\right|\Big\|_p&\lesssim \|m\|_{\Sigma^2\left(B_{p_0}^{d/p_0+\alpha+2\varepsilon}\right)}\|f\|_p\,.
\end{align*}
Since $\frac{d}{p_0}+\alpha<s$, and $\varepsilon>0$ is arbitrary,
we obtain the desired results.

It remains to prove \eqref{ineq_L2esti}, \eqref{ineq_endpts}, and Lemma~\ref{lem_emb} to complete the proof.
We present each proof in following subsections.

\subsection{Estimate \eqref{ineq_L2esti}}
We first note that due to $\kappa(E)<1$ and Lemma~\ref{231109742}, it follows from $\beta>\frac{\kappa(E)}{2}$ that $N(\widetilde{E}_j,\delta)\leq N(E_j,\delta)+2\leq C_\beta \delta^{-2\beta}$, and therefore
\begin{align}\label{2404161012}
\sup_{j\in\bZ}\int_1^2d(s,\widetilde{E}_j)^{-1+2\beta}\dd s<\infty\,.
\end{align}

For fixed $\xi\in\bR^d$, take $k_0\in\bZ$ such that $2^{k_0}\leq |\xi|\leq 2^{k_0+1}$.
Then 
\begin{align*}
&\sum_{j\in\bZ}\int_1^2d(s,E_j)^{-1+2\beta}\big|m(2^js\xi)\big|^2\dd s\\
\lesssim\,&\sum_{j\in\bZ}\sum_{k=0,1,2}\int_1^2d(s,E_j)^{-1+2\beta}\dd s\, \|m(2^{j+k_0}\,\cdot\,)\widehat{\psi}_k\|_\infty^2\\
\lesssim\,& \sum_{j\in\bZ}\|m(2^j\,\cdot\,)\widehat{\psi}\|_\infty^2\,,
\end{align*}
where the first inequality follows from that
$$
\sum_{k=0,1,2}\widehat{\psi}_k\equiv 1\quad \text{on}\quad \{\xi\in\bR^d\,:\,1\leq |\xi|\leq 4\}
$$
(note that $\widehat{\psi}_j$ denotes the Littlewood--Paley decomposition defined in \eqref{psi}),
and the second inequality is implied by \eqref{2404161012}.
Thus we obtain
$$
\|M_m\|_{L_{\infty}(\bR^d;\cH)}\lesssim \left(\sum_{j\in\bZ}\|m(2^j\,\cdot\,)\widehat{\psi}\|_\infty^2\right)^{1/2}\,,
$$
which implies
\begin{align*}
    \|G(m,f)\|_{L^2(\mathcal{H})\to L^2(\mathcal{H})}
    \lesssim \left(\sum_{j\in\bZ}\|m(2^j\,\cdot\,)\widehat{\psi}\|_\infty^2\right)^{1/2}\,.
\end{align*}

\subsection{Estimate \eqref{ineq_endpts}}

For any $k\in\bZ$,  
\begin{align*}
&\|M_m(2^k\,\cdot\,)\widehat{\psi}\|_{L^2(\bR^d;\cH)}^2\\
=&\int_{\bR^d}\left(\sum_{j\in\bZ}\int_1^2d(s,E_j)^{-1+2\beta}\big|m\big(2^js(2^k\xi)\big)\widehat{\psi}(\xi)\big|^2\dd s\right)\dd \xi\\
\lesssim & \int_{\bR^d}\left(\sum_{j\in\bZ}\int_1^2d(s,E_j)^{-1+2\beta}(2^{j+k}s)^{-d}\left|\widehat{\psi}\Big(\frac{\xi}{2^{j+k}s}\Big)\right|^2\dd s\right) |m(\xi)|^2\dd \xi\,.
\end{align*}
One can observe that for $1\leq s \leq 2$, 
$$
\left(2^{j+k}s\right)^{-d}\left|\widehat{\psi}\left(\frac{\xi}{2^{j+k}s}\right)\right|^2\lesssim |\xi|^{-d}1_{4^{-1}\leq |\xi|/(2^{j+k}s)\leq 4}\leq |\xi|^{-d}1_{8^{-1}\leq |\xi|/(2^{j+k})\leq 8}\,.
$$
This implies that
\begin{align*}
&\sum_{j\in\bZ}\int_1^2d(s,E_j)^{-1+2\beta}(2^{j+k}s)^{-d}\left|\widehat{\psi}\left(\frac{\xi}{2^{j+k}s}\right)\right|^2\dd s\\
\lesssim\, &\sum_{j\in\bZ}\left(\int_1^2d(s,E_j)^{-1+2\beta}\dd s\right) 1_{8^{-1}\leq |\xi|/(2^{j+k})\leq 8}|\xi|^{-d}\\
\lesssim\, &\sum_{j\in\bZ}1_{8^{-1}\leq |\xi|/(2^{j+k})\leq 8}|\xi|^{-d}\\
\lesssim\,& |\xi|^{-d}\,,
\end{align*}
where the second inequality follows from \eqref{2404161012}.
Consequently, we have
\begin{align}\label{2311091104}
\sup_{k\in\bZ}\|M_m(2^k\,\cdot\,)\widehat{\psi}\|_{L^2(\bR^d;\cH)}\lesssim \left(\int_{\bR^d}|m(\xi)|^2\frac{\mathrm{d}\xi}{|\xi|^d}\right)^{1/2}\simeq \|m\|_{\Sigma(L^2)}\,.
\end{align}
Note \eqref{2311091104} holds for all $\widehat{\psi}\in C_c^{\infty}(B_4\setminus\overline{B}_{1/4})$.

For any $l\in\bN$,
$$
\|M_m(2^k\,\cdot\,)\widehat{\psi}\|_{L^2_l(\bR^d;\cH)}\lesssim \sum_{|\gamma|+|\gamma'|=l}\|2^{k|\gamma|}\big(\partial^\gamma M_m\big)(2^k\,\cdot\,)\partial^{\gamma'}\widehat{\psi}\|_{L^2(\bR^d;\cH)}\,.
$$
Since $M_m(\xi)=\{m(2^js\xi)\}_{j,s}$, we have
\begin{align}\label{231091119}
\begin{split}
\big(\partial^\gamma M_m\big)(\xi)&=\big\{\big(2^{j}s\big)^{|\gamma|}\times \big(\partial^\gamma m\big)(2^js\xi)\big\}_{j,s}\\
&=\big\{\big(|\cdot|^{|\gamma|}\partial^\gamma m\big)(2^js\xi)\times |\xi|^{-n}\big\}_{j,s}\\
&=M_{|\,\cdot\,|^{|\gamma|}\partial^\gamma m}(\xi)\times |\xi|^{-n}\,.
\end{split}
\end{align}
Due to \eqref{2311091104}, \eqref{231091119}, and Lemma~\ref{231109931}, we have
\begin{align}\label{2311091107}
\|M_m(2^k\,\cdot\,)\widehat{\psi}\|_{L^2_l(\bR^d;\cH)}&\lesssim \sum_{|\gamma|+|\gamma'|\leq l}\|2^{k|\gamma|}\big(\partial^\gamma M_m\big)(2^k\,\cdot\,)\partial^{\gamma'}\widehat{\psi}\|_{L^2(\bR^d;\cH)}\nonumber\\
&=\sum_{|\gamma|+|\gamma'|\leq l}\bigg\|M_{|\,\cdot\,|^{|\gamma|}\partial^{\gamma}m}(2^k\,\cdot\,)\times \frac{\partial^{\gamma'}\widehat{\psi}}{|\cdot|^{|\gamma|}}\bigg\|_{L^2(\bR^d;\cH)}\\
&\lesssim \sum_{|\gamma|\leq l} \int_{\bR^d}\big|\partial^{\gamma}m(\xi)\big|^2|\xi|^{2|\gamma|-d}\dd \xi\nonumber\\
&\simeq \|m\|_{\Sigma^2(L^2_l)}\nonumber
\end{align}
(for the last similarity, see Lemma~\ref{231109931}.(2)).

By interpolating \eqref{2311091107}, we obtain that for any $\beta\geq 0$, 
\begin{align*}
\sup_{k\in\bZ}\|M_m(2^k\,\cdot\,)\widehat{\psi}\|_{L^2_\beta(\bR^d;\cH)}\lesssim\|m\|_{\Sigma^2(L^2_\beta)}
\end{align*}
(see \cite[Theorem~5.6.9]{HNVW2016} for the interpolation of vector-valued Sobolev spaces $L^2_\beta(\bR^d;\cH)$, and see Lemma~\ref{231109931}.(4) for the interpolation of $\Sigma^2(L^2_\beta)=\Sigma^2(B_2^\beta)$).

\subsection{Proof of Lemma~\ref{lem_emb}}
Recall that
$$
\wt{m}(\xi)=\int_0^1\frac{m(\xi)-m(r\xi)}{(1-r)^{1+\alpha}}\dd r=\int_0^{1/2}(\,\cdots\,)\dd r+\int_{1/2}^{1}(\,\cdots\,)\dd r=:\wt{m}_1+\wt{m}_2\,.
$$

It follows from Lemma~\ref{231109931}.(2) that
\begin{align}\label{231109939}
\|m-m(r\,\cdot\,)\|_{\Sigma^2(B_p^s)}\leq \|m\|_{\Sigma^2(B_p^s)}+\|m(r\,\cdot\,)\|_{\Sigma^2(B_p^s)}\lesssim \|m\|_{\Sigma^2(B_p^s)}\,,
\end{align}
which implies 
\begin{align}
\|\wt{m}_1\|_{\Sigma^2(B_p^s)}\,&\leq \int_0^{1/2}\frac{\|m-m(r\,\cdot\,)\|_{\Sigma^2(B_p^s)}}{(1-r)^{1+\alpha}}\dd r\nonumber\\
&\lesssim \int_0^{1/2}\frac{1}{(1-r)^{1+\alpha}}\dd r\times \|m\|_{\Sigma^2(B_p^s)}\label{231109947}\\
&\lesssim\|m\|_{\Sigma^2(B_p^{s+\alpha+\varepsilon})}\,.\nonumber
\end{align}

To estimate $\|\wt{m}_2\|_{\Sigma^2(B_p^s)}$, observe that
$$
m(\xi)-m(r\xi)=\int_r^1\xi\cdot\nabla m(t\xi)\dd t\,,
$$
which implies that for any $r\in\left(\frac{1}{2},1\right)$,
\begin{align}
\|m-m(r\,\cdot\,)\|_{\Sigma^2(B_p^s)}&\leq \int_r^1 t^{-1}\|t\xi\cdot \nabla m(t\xi)\|_{\Sigma^2(B_p^s)}\dd t\nonumber\\
&\lesssim \int_r^1 t^{-1}\dd t\|\xi\cdot \nabla m(\xi)\|_{\Sigma^2(B_p^s)}\label{231109938}\\
&\lesssim (1-r)\|m\|_{\Sigma^2(B_p^{s+1})}\,,\nonumber
\end{align}
where the inequalities are implied by Lemma~\ref{231109931}.(3).
By interpolating \eqref{231109938} and \eqref{231109939} (see Lemma~\ref{231109931}.(4)), we obtain that for any $\varepsilon\in(0,1-\alpha)$ and $r\in\left(\frac{1}{2},1\right)$, 
\begin{align}\label{231109943}
\|m-m(r\,\cdot\,)\|_{\Sigma^2(B_p^s)}\leq  C\, (1-r)^{\alpha+\varepsilon}\|m\|_{\Sigma^2(B_p^{s+\alpha+\varepsilon})}\,,
\end{align}
where $C=C(d,s,p,\alpha,\varepsilon)$.
Due to \eqref{231109943}, we have
\begin{equation}\label{231109946}
\begin{aligned}
\|\wt{m}_2\|_{\Sigma^2(B_p^s)}&\leq \int_{1/2}^1\frac{\|m-m(r\,\cdot\,)\|_{\Sigma^2(B_p^s)}}{(1-r)^{1+\alpha}}\dd r\\
&\lesssim \int_{1/2}^{1}(1-r)^{-1+\varepsilon}\dd r\times \|m\|_{\Sigma^2(B_p^{s+\alpha+\varepsilon})}\,.
\end{aligned}
\end{equation}

Since $\int_{1/2}^1(1-r)^{-1+\varepsilon}\dd r<\infty$, by combining \eqref{231109946} and \eqref{231109947}, we have
$$
\|\wt{m}\|_{\Sigma^2(B_p^s)}\leq \|\wt{m}_1\|_{\Sigma^2(B_p^s)}+\|\wt{m}_2\|_{\Sigma^2(B_p^s)}\lesssim \|m\|_{\Sigma^2(B_p^{s+\alpha+\varepsilon})}\,,
$$
and the proof is completed.

\section*{Acknowledgement}
J. B. Lee has been supported by the National Research Foundation of Korea(NRF) grants funded by the Korea government(MSIT) (No. 2021R1C1C\-2008252, No. 2022R1A4A1018904).
J. Seo has been supported by a KIAS Individual Grant (MG095801) at Korea Institute for Advanced Study, and by the National Research Foundation of Korea (NRF-2019R1A5A1028324).

\bibliographystyle{acm}

\end{document}